\documentclass{article}
\usepackage{caption}
\usepackage{graphicx}
\usepackage[utf8]{inputenc}
\usepackage{amssymb,amsthm}

\usepackage[dvipsnames]{xcolor}
\usepackage[shortlabels]{enumitem}
\usepackage[utf8]{inputenc}
\usepackage{amssymb,amsmath, amsfonts}
\usepackage{mathtools}
\usepackage{hyperref}
\usepackage[portuguese,english]{babel}
\usepackage{xcolor}
\theoremstyle{plain}

\newtheorem{Theorem}{Theorem}
\newtheorem{Lemma}{Lemma}
\newtheorem{Claim}{Claim}
\newtheorem{case}{Case}
\newtheorem{Proposition}{Proposition}
\newtheorem{Problem}{Problem}

\newtheorem{Conjecture}{Conjecture}
\theoremstyle{nonumberplain}

\theoremstyle{nonumberplain}
\title{Globally balancing spanning trees}
\author{Florian H\"orsch}

\newcommand{\stf}{spanning tree factorization }
\begin{document}
\maketitle
\begin{abstract}
We show that for every graph $G$ that contains two edge-disjoint spanning trees, we can choose two edge-disjoint spanning trees $T_1,T_2$ of $G$ such that $|d_{T_1}(v)-d_{T_2}(v)|\leq 5$ for all $v \in V(G)$. We also prove the more general statement that for every positive integer $k$, there is a constant $c_k \in O(\log k)$ such that for every graph $G$ that contains $k$ edge-disjoint spanning trees, we can choose $k$ edge-disjoint spanning trees $T_1,\ldots,T_k$ of $G$ satisfying $|d_{T_i}(v)-d_{T_j}(v)|\leq c_k$ for all $v \in V(G)$ and $i,j \in \{1,\ldots,k\}$. This resolves a conjecture of Kriesell.
\end{abstract}

\section{Introduction}
This paper deals with the factorization of graphs into spanning trees that satisfy certain extra conditions. All undefined notations can be found in Section \ref{prel}.

 In 1961, Nash-Williams \cite{nw} and Tutte \cite{t} independently proved a fundamental theorem characterizing the graphs which contain a fixed number of edge-disjoint spanning trees. Since then, several results have been found which extend this theorem. For example, while it does not extend to infinite graphs in general, it has been generalized to certain classes of infinite graphs by Lehner \cite{l} and Stein \cite{s}. Also, a packing of spanning trees of a given size of minimum weight with respect to a given weight function on the edges can be computed efficiently using matroid theory.

 The extensions we are most interested in are those that impose further restrictions on the spanning trees we wish to pack. Here spanning trees of bounded diameter have been considered by Chuzoy, Parter and Tan \cite{cpt}. Also, spanning trees which satisfy a certain equilibrium condition stating that the deletion of a designated vertex should not leave a graph in which almost all vertices are contained in a single connected component have been considered by Bang-Jensen, Havet and Yeo \cite{bhy} and Bessy et al. \cite{bhmrs}.

In this article, we deal with a different balance condition imposed on the spanning trees which was first considered by Kriesell \cite{nk1}. We wish to know whether we can find a packing of spanning trees of a given graph $G$ such that for every $v \in V(G)$ the degree of $v$ is roughly the same in each of the spanning trees. There is one remarkable difference between the problems considered by Chuzoy et al., Bang-Jensen et al. and Bessy et al. and the class of problems considered in this article. For the former problems, it is not difficult to see that there are graphs that admit a packing of spanning trees of a given size but do not admit such a packing if the extra condition is imposed. In our problem, we wish to understand whether for every graph that contains a packing of spanning trees of a certain size, these spanning trees can be chosen so to satisfy the balance condition.

Clearly, when trying to give a positive answer to this question, we may restrict to graphs whose edge set only consists of the union of the edge sets of the spanning trees in the packing as additional edges can only make the task easier. More formally, we say that a graph $G$ is a {\it $k$-multiple tree} if there is a collection of spanning trees $(T_1,\ldots,T_k)$ of $G$ such that $E(T_i)\cap E(T_j)=\emptyset$ for all $i,j \in \{1,\ldots,k\}$ with $i \neq j$ and $\bigcup_{i \in \{1,\ldots,k\}}E(T_i)=E(G)$. We abbreviate a $2$-multiple tree to a {\it double tree}. Further, we call $(T_1,\ldots,T_k)$ a {\it \stf} of $G$. By the remark above, all results on balanced spanning tree packings proven for $k$-multiple trees also hold more generally for graphs containing $k$ edge-disjoint spanning trees. For the sake of simplicity, we give them in the form restricted to $k$-multiple trees.


In \cite{nk1}, Kriesell focused on the case $k=2$ and showed that a rather mild balance condition can always be achieved. Namely, he proved the following result showing that every double tree has a \stf such that a vertex that is  a leaf in one of the spanning trees in the factorization is of bounded degree in the whole graph. 
\begin{Theorem} \cite{nk1}\label{feuille}
Let $G$ be a double tree. Then $G$ has a \stf $(T_1,T_2)$ such that every vertex $v \in V(G)$ that is a leaf of one of $T_1$ and $T_2$ satisfies $d_G(v)\leq 8$. 
\end{Theorem} 
He also asked whether the constant in Theorem \ref{feuille} can be improved.

\begin{Problem}\cite{nk1} \label{dggfd}
What is the minimum integer $c$ for which every double tree $G$ has a \stf $(T_1,T_2)$ such that every vertex $v \in V(G)$ that is a leaf of one of $T_1$ and $T_2$ satisfies $d_G(v)\leq c$?
\end{Problem}

Theorem \ref{feuille} shows that $c \leq 8$. Kriesell also gave a construction in \cite{nk1} showing that $c \geq 4$.

Further, he posed the following much more general conjecture.

\begin{Conjecture}\cite{nk1} \label{ddfd}
For every integer $k$, there exists an integer $c_k$ such that every $k$-multiple tree has a \stf $(T_1,\ldots,T_k)$ satisfying $|d_{T_i}(v)-d_{T_j}(v)|\leq c_k$ for all $v \in V(G)$ and $i,j \in \{1,\ldots,k\}$.
\end{Conjecture}

Our main contribution is the confirmation of Conjecture \ref{ddfd}.
\medskip

 First, we show the following result that solves Conjecture \ref{ddfd} for $k=2$.

\begin{Theorem}\label{main}
Let $G$ be a double tree. Then there is a spanning tree factorization $(T_1,T_2)$ of $G$ such that $|d_{T_1}(v)-d_{T_2}(v)|\leq 5$ for all $v \in V(G)$.
\end{Theorem}

Observe that Theorem \ref{main} improves on Theorem \ref{feuille} in two ways. Apart from confirming Conjecture \ref{ddfd} for $k=2$, it also implies that the constant $c$ in Problem \ref{dggfd} is at most 7.

Surprisingly, while the confirmation of Conjecture \ref{ddfd} for $k=2$ seems to be the much stronger of these two achievements, this can actually be obtained by arguments which are similar to those used by Kriesell in \cite{nk1} to prove Theorem \ref{feuille}. More concretely, a version of Theorem \ref{main} where the constant $5$ is replaced by $6$ can be obtained by altering the arguments in \cite{nk1}. On the other hand, for the improvement of this constant some more sophisticated refinements are required. Our proof of Theorem \ref{main} is based on a discharging argument relying on the fact that every double tree $G$ satisfies $|E(G)|=2|V(G)|-2$. 
\medskip

After, we confirm Conjecture \ref{ddfd} in general. More precisely, we show the following result.

\begin{Theorem}\label{newmain}
Let $G$ be a $k$-multiple tree. Then there is a \stf $(T_1,\ldots,T_k)$ of $G$ such that $|d_{T_i}(v)-d_{T_j}(v)|\leq 22\log(k)$ for all $v \in V(G)$ and $i,j \in \{1,\ldots,k\}$.
\end{Theorem}

The main idea of the proof of Theorem \ref{newmain} is to start with an arbitrary \stf of $G$ and then repeatedly improve this \stf by applying Theorem \ref{main} a significant number of times to pairs of spanning trees contained in the packing. This method allows to gradually obtain more and more balanced spanning tree factorizations. 
\medskip

In Section \ref{prel}, we formally define our notation and collect some preliminary results. In Section \ref{alt}, we prove Theorem \ref{main}.  In Section \ref{neu}, we prove Theorem \ref{newmain}. Finally, we conclude our work and give some directions for further research in Section \ref{conc}.
\section{Preliminaries}\label{prel}

We first give some basic notions of graph theory. We often use $x$ for a single element set $\{x\}$. For a graph $G$, we let $V(G)$ and $E(G)$ denote the vertex set and the edge set of $G$, respectively. For some edge $e=uv \in E(G)$, we say that $u$ and $v$ are {\it incident} to $e$ and that $e$ {\it links} $u$ and $v$. For some $X \subseteq V$, we use $\delta_G(X)$ to denote the set of edges in $E(G)$ that link a vertex in $X$ and a vertex in $V(G)-X$. We use $d_G(X)$ for $|\delta_G(X)|$. For some $v \in V(G)$, we define $G-v$ by $V(G-v)=V(G)-v$ and $E(G-v)=E(G)-\delta_G(v)$ and for some $e \in E(G)$, we define $G-e$ by $V(G-e)=V(G)$ and $E(G-e)=E(G)-e$. Further, for a new edge $e=uv \notin E(G)$ with $u,v \in V(G)$ we define $G+e$ by $V(G+e)=V(G)$ and $E(G+e)=E(G)+e$ and for a new edge $e=uv \notin E(G)$ with $u \in V(G)$ and $v \notin V(G)$ we define $G+e$ by $V(G+e)=V(G)+v$ and $E(G+e)=E(G)+e$. A graph $H$ is called a {\it subgraph} of another graph $G$ if $V(H)\subseteq V(G)$ and $E(H)\subseteq E(G)$.

A {\it tree} is a connected graph that does not contain any cycles. Given a graph $G$, a subgraph $T$ of $G$ is called a {\it spanning tree} of $G$ if $T$ is a tree satisfying $V(T)=V(G)$. Recall that if a graph $G$ contains $k$ edge-disjoint spanning trees $T_1,\ldots, T_k$ with $E(T_1)\cup \ldots \cup E(T_k)=E(G)$, we say that $G$ is a {\it $k$-multiple tree} and $(T_1,\ldots,T_k)$ is a {\it \stf}\ of $G$. Further recall that a $2$-multiple tree is also called a {\it double tree}.

We now give some preliminary results we need for our proof of Theorem \ref{main}. The following statement is an immediate consequence of a well-known property of trees. It will be crucial for the discharging procedure in Section \ref{alt}.
\begin{Proposition}\label{aretes}
Every double tree satisfies $|E(G)|=2|V(G)|-2$.
\end{Proposition}
The following simple observation shows that when checking if a \stf satisfies the conclusion of Theorem \ref{main}, it suffices to do so for vertices of large degree. This fact will be used frequently in Section \ref{alt}.
\begin{Proposition}\label{trivial}
Let $G$ be a double tree and $(T_1,T_2)$ a \stf of $G$ such that $|d_{T_1}(v)-d_{T_2}(v)|\leq 5$ for all $v \in V(G)$ with $d_G(v)\geq 8$. Then $|d_{T_1}(v)-d_{T_2}(v)|\leq 5$ for all $v \in V(G)$.
\end{Proposition}
\begin{proof}
It suffices to prove that $|d_{T_1}(v)-d_{T_2}(v)|\leq 5$ for all $v \in V(G)$ with $d_G(v)\leq 7$. Observe that as $T_1$ is a spanning tree of $G$, we have that $d_{T_1}(v)\geq 1$. As $(T_1,T_2)$ is a \stf of $G$, we obtain $d_{T_2}(v)=d_G(v)- d_{T_1}(v)\leq 7-1=6$. This yields $d_{T_2}(v)-d_{T_1}(v)\leq 6-1=5$. Similarly, we have $d_{T_1}(v)-d_{T_2}(v)\leq 5$, so $|d_{T_1}(v)-d_{T_2}(v)|\leq 5$.
\end{proof}
The following well-known exchange property of spanning trees can be found in a stronger form as Theorem 5.3.3  in \cite{book}.
\begin{Proposition}\label{map}
Let $G$ be a double tree and let $(T_1,T_2)$ be a \stf of $G$. Then there is a function $\sigma:E(T_1)\rightarrow E(T_2)$ such that $(T_1-e+\sigma(e),T_2-\sigma(e)+e)$ is a \stf of $G$ for all $e \in E(T_1)$.
\end{Proposition}

We call a function like in Proposition \ref{map} a {\it tree-mapping} function from $T_1$ to $T_2$. We now show two important properties of tree-mapping functions that will be useful in Section \ref{alt}.
\begin{Proposition}\label{trivial2}
Let $G$ be a double tree, $(T_1,T_2)$ a \stf of $G$ and $x \in V(G)$ such that $x$ is incident to a unique edge $e$ in $T_1$. Then for any tree-mapping function $\sigma:E(T_1)\rightarrow E(T_2)$, we have $\sigma(e)\in \delta_G(x)$.
\end{Proposition}
\begin{proof}
As $T_1-e+\sigma(e)$ is a spanning tree of $G$, there is at least one edge $f \in( E(T_1)-e+\sigma(e))\cap \delta_G(x)$. As $(E(T_1)-e)\cap \delta_G(x)= \emptyset$ by assumption, we have $f=\sigma(e)$.
\end{proof}
\begin{Proposition}\label{triangle}
Let $G$ be a double tree, $(T_1,T_2)$ a \stf of $G$ and $xyz$ a triangle in $G$ such that $xy,xz \in E(T_1)$ and $yz \in E(T_2)$. Then for any tree-mapping function $\sigma:E(T_1)\rightarrow E(T_2)$, we have either $\sigma(xy)\neq yz$ or $\sigma(xz)\neq yz$.
\end{Proposition}
\begin{proof}
Observe that $T_2-yz$ has exactly two connected components $C_y,C_z$ with $y \in V(C_y)$ and $z \in V(C_z)$. By symmetry, we may suppose that $x \in V(C_y)$. This yields that no edge leaves $V(C_z)$ in $T_2-yz+xy$, so $T_2-yz+xy$ is not a spanning tree of $G$. As $\sigma$ is a tree-mapping function, we obtain that $\sigma(xy)\neq yz$.
\end{proof}

For the proof of Conjecture \ref{ddfd} in Section \ref{neu}, we need some basic results of analysis. By $\log$, we refer to the logarithm for the basis $2$. We denote the set of nonnegative integers by $\mathbb{N}$. For all other basic notions, see \cite{w}. The first result we need is well-known.
\begin{Proposition}\label{gzigzi}
Let $c_1,c_2 \in \mathbb{R}$ with $0<c_1<1$ and $(x_n)_{n \in \mathbb{N}}$ a series of real numbers satisfying $x_n=c_1x_{n-1}+c_2$ for all $n \geq 1$. Then $(x_n)_{n \in \mathbb{N}}$ converges and $\lim_{n \rightarrow \infty}x_n=\frac {c_2}{1-c_1}$.
\end{Proposition}

We next  need the following simple result which we prove for the sake of completeness.
\begin{Proposition}\label{ddaqd}
Let $(x_n)_{n \in \mathbb{N}}$ be a converging series of reals and $x=\lim_{n \rightarrow \infty}x_n$. Further, let $(d_n)_{n \in \mathbb{N}}$ be a series of reals such that $d_n \geq x_n$ and $d_n \in S$ for some finite set $S$ for all $n \in \mathbb{N}$. Then there is some $n_0 \in \mathbb{N}$ such that $d_n \geq x$ for all $n \geq n_0$.
\end{Proposition}
\begin{proof}
Let $\bar{s}=\max\{s \in S:s<x\}$. As $\lim_{n \rightarrow \infty}x_n=x$, there is some $n_0 \in \mathbb{N}$ such that $x_n > \bar{s}$ for all $n \geq n_0$. For all $n \geq n_0$, we obtain $d_n\geq x_n > \bar{s}$. By the choice of $\bar{s}$ and $d_n \in S$ for all $n \in \mathbb{N}$, we obtain $d_n \geq x$ for all $n \geq n_0$.
\end{proof}
We finally need the following very basic observation that can easily be verified.
\begin{Proposition}\label{loga}
Let $\mu$ be a positive integer. Then $\frac{11 \log (\mu)+\frac{5}{2}}{2 \mu +1}\leq \frac{7}{2}$.
\end{Proposition}
\section{Balancing two spanning trees}\label{alt}

This section is dedicated to proving Theorem \ref{main}.
\medskip

 For the sake of a contradiction, we suppose that Theorem \ref{main} is wrong and choose a counterexample $G$ such that $|V(G)|$ is minimum. Throughout the proof, we call a vertex $v \in V(G)$ with $d_G(v)=i$ for some integer $i$ an {\it $i$-vertex}. A vertex $v \in V(G)$ is called {\it big} if $d_G(v)\geq 8$ and {\it small} otherwise. For some integers $i,j$, an edge $e \in E(G)$  is called an {\it $(i,j)$-edge} if it links an $i$-vertex and a $j$-vertex and an {\it $(i,big)$-edge} if it links an $i$-vertex and a big vertex.
\medskip

 We first give some structural properties of $G$ in Section \ref{struc}. After, we obtain a contradiction using a discharging procedure in Section \ref{charge}.

\subsection{Structural properties of a minimum counterexample}\label{struc}

In this section, we collect some structural properties of $G$. We divide this into three parts. In Section \ref{base}, we give a collection of properties of $G$ which are similar to the ones found by Kriesell in \cite{nk1}. While many of the ideas can be found in \cite{nk1}, we alter the results in order to make them helpful for proving the more general statement of Theorem \ref{main}. After simple adapations, the results in Section \ref{base} would already suffice to prove a version of Theorem \ref{main} in which the constant $5$ is replaced by $6$. This could be done using a discharging procedure which is similar but simpler than the one used in Section \ref{charge}.

Some more care is needed for the improvement of this constant. The main new result which is used in the discharging procedure is Lemma \ref{new}. In Section \ref{tech1}, we give two slightly technical preparatory results that show that some degenerate cases which could cause problems for the proof of Lemma \ref{new} do not actually occur. After, in Section \ref{hstrrh}, we give the proof of Lemma \ref{new}.
\subsubsection{Basic structural properties}\label{base}
We start by dealing with 2-vertices. We first exclude one degenerate case.

\begin{Proposition}\label{dz}
No 2-vertex is incident to two parallel edges.
\end{Proposition}
\begin{proof}
Suppose for the sake of a contradiction that there is a 2-vertex $x$ which is linked by two edges $e,f$ to another vertex $y$ in $G$. Let $(T_1,T_2)$ be a \stf of $G$. Clearly, one of $e$ and $f$ is contained in $E(T_1)$, while the other one is contained in $E(T_2)$. Let $G'=G-x, T_1'=T_1-x$  and $T_2'=T_2-x$. We obtain that $(T_1',T_2')$ is a \stf of $G'$, so $G'$ is a double tree. As $G'$ is smaller than $G$, we obtain that $G'$ has a \stf $(S_1',S_2')$ such that $|d_{S'_1}(v)-d_{S'_2}(v)|\leq 5$ for all $v \in V(G')$. Now let $S_1=S_1'+e$ and $S_2=S_2'+f$.

 Clearly, $(S_1,S_2)$ is a \stf of $G$. We have $|d_{S_1}(v)-d_{S_2}(v)|=|d_{S'_1}(v)-d_{S'_2}(v)|\leq 5$ for all $v \in V(G)-\{x,y\}$. Next, we have $|d_{S_1}(y)-d_{S_2}(y)|=|(d_{S'_1}(y)+1)-(d_{S'_2}(y)+1)|=|d_{S'_1}(y)-d_{S'_2}(y)|\leq 5$. As $x$ is small, by Proposition \ref{trivial}, we obtain a contradiction to $G$ being a counterexample.
\end{proof}

 The following result shows that no 2-vertex can be linked to another small vertex in $G$.
\begin{Lemma}\label{2gross}
Every 2-vertex is incident to two $(2,big)$-edges.
\end{Lemma}
\begin{proof}
Let $x$ be a 2-vertex and let $xy,xz$ be the two edges incident to $x$. By Proposition \ref{dz}, we have $y \neq z$. Suppose for the sake of a contradiction that one of $y$ and $z$, say $y$, is small. Let $(T_1,T_2)$ be a \stf of $G$. As $T_1$ and $T_2$ are spanning trees, each of them has to contain exactly one edge incident to $x$. Let $G'=G-x, T_1'=T_1-x$  and $T_2'=T_2-x$. We obtain that $(T_1',T_2')$ is a \stf of $G'$, so $G'$ is a double tree. As $G'$ is smaller than $G$, we obtain that $G'$ has a \stf $(S_1',S_2')$ such that $|d_{S'_1}(v)-d_{S'_2}(v)|\leq 5$ for all $v \in V(G')$. By symmetry, we may suppose that $d_{S'_1}(z) \geq d_{S'_2}(z)$. Let $S_1=S'_1+xy$ and $S_2=S_2'+xz$.

 Clearly, $(S_1,S_2)$ is a \stf of $G$. We have $|d_{S_1}(v)-d_{S_2}(v)|=|d_{S'_1}(v)-d_{S'_2}(v)|\leq 5$ for all $v \in V(G)-\{x,y,z\}$. If $d_{S'_1}(z) = d_{S'_2}(z)$, we obtain $|d_{S_1}(z)-d_{S_2}(z)|=|d_{S'_1}(z)-(d_{S'_2}(z)+1)|=1 < 5$. Otherwise, as $d_{S'_1}(z) > d_{S'_2}(z)$, we obtain $|d_{S_1}(z)-d_{S_2}(z)|=|d_{S'_1}(z)-d_{S'_2}(z)|-1\leq 4 < 5$.  As $x$ and $y$ are small, by Proposition \ref{trivial}, we obtain a contradiction to $G$ being a counterexample.
\end{proof}
We next wish to show that no big vertex can be linked to several 2-vertices. 

\begin{Lemma}\label{only}
Every big vertex is incident to at most one $(2,big)$-edge.
\end{Lemma}
\begin{proof}
Suppose for the sake of a contradiction that there is a big vertex $x$ that is incident to two $(2,big)$-edges $xy_1$ and $xy_2$. By Proposition \ref{dz}, we have $y_1 \neq y_2$. Let $z_1,z_2$  be the neighbors of $y_1$ and $y_2$ different from $x$, respectively. By Lemma \ref{2gross}, we have $z_1 \neq y_2$ and $z_2 \neq y_1$. Note that possibly $z_1=z_2$. We now create the graph $G'$ from $G$ by deleting $y_1$ and $y_2$, adding a new vertex $y$ and adding the two edges $z_1y$ and $z_2y$. 
\begin{Claim}\label{2t}
$G'$ is a double tree.
\end{Claim}
\begin{proof}[Proof of Claim]
Let $(T_1,T_2)$ be a \stf of $G$. As $y_1$ and $y_2$ are 2-vertices, we obtain that $(T_1-\{y_1,y_2\}, T_2-\{y_1,y_2\})$ is a \stf of $G-\{y_1,y_2\}$. It follows that $(T_1-\{y_1,y_2\}+z_1y, T_2-\{y_1,y_2\}+z_2y)$ is a \stf of $G'$, so $G'$ is a double tree.
\renewcommand{\qedsymbol}{$\blacksquare$}
\end{proof}
By Claim \ref{2t} and as $G'$ is smaller than $G$,  we obtain that $G'$ has a \stf $(S_1',S_2')$ such that $|d_{S'_1}(v)-d_{S'_2}(v)|\leq 5$ for all $v \in V(G')$. By symmetry and as $S_1'$ and $S_2'$ are spanning trees of $G'$, we may suppose that $z_1y \in E(S'_2)$ and $z_2y \in E(S'_1)$. Now let $S_1=S_1'-y+\{xy_1,y_2z_2\}$ and $S_2=S_2'-y+\{xy_2,y_1z_1\}$.

 Clearly, $(S_1,S_2)$ is a \stf of $G$. We have $|d_{S_1}(v)-d_{S_2}(v)|=|d_{S'_1}(v)-d_{S'_2}(v)|\leq 5$ for all $v \in V(G)-\{x,y_1,y_2\}$. Next, we have $|d_{S_1}(x)-d_{S_2}(x)|=|(d_{S'_1}(x)+1)-(d_{S'_2}(x)+1)|=|d_{S'_1}(x)-d_{S'_2}(x)|\leq 5$. As $y_1$ and $y_2$ are small, by Proposition \ref{trivial}, we obtain a contradiction to $G$ being a counterexample.
\end{proof}

We now deal with 3-vertices. Observe that given a 3-vertex $x$ and a \stf $(T_1,T_2)$ of $G$, we have that $x$ is of degree $1$ in $T_i$ and of degree $2$ in $T_{3-i}$ for some $i \in \{1,2\}$. We then call the unique edge in $\delta_G(x)\cap E(T_i)$ the {\it special} edge of $x$ with respect to $(T_1,T_2)$.

In Proposition \ref{cfzz} we show that every 3-vertex is linked to a big vertex by its special edge. Proposition \ref{cfzz} will also be applied in Section \ref{hstrrh}. After, in Lemma \ref{aumoins23}, we use this to conclude that every 3-vertex is incident to at least two $(3,big)$-edges. 
\begin{Proposition}\label{cfzz}
Let $(T_1,T_2)$ be a \stf of $G$ and let $x,y \in V(G)$ such that $x$ is a 3-vertex and $xy$ is the special edge of $x$ with respect to $(T_1,T_2)$. Then $y$ is big.
\end{Proposition}
\begin{proof}
Let $x$ be incident to the three edges $xy,xz_1$ and $xz_2$. By assumption and symmetry, we may suppose that $xy \in E(T_1)$ and $xz_1,xz_2 \in E(T_2)$. Now let $G'=G-x+z_1z_2, T_1'=T_1-x$ and $T_2'=T_2-x+z_1z_2$.  We obtain that $(T_1',T_2')$ is a \stf of $G'$, so $G'$ is a double tree. As $G'$ is smaller than $G$, we obtain that $G'$ has a \stf $(S_1',S_2')$ such that $|d_{S'_1}(v)-d_{S'_2}(v)|\leq 5$ for all $v \in V(G')$. By symmetry, we may suppose that $z_1z_2 \in E(S'_2)$. Now let $S_1=S_1'+xy$ and $S_2=S_2'-z_1z_2+\{xz_1,xz_2\}$.

 Clearly, $(S_1,S_2)$ is a \stf of $G$. We have $|d_{S_1}(v)-d_{S_2}(v)|=|d_{S'_1}(v)-d_{S'_2}(v)|\leq 5$ for all $v \in V(G)-\{x,y\}$. As $x$ is small, by Proposition \ref{trivial}, if $y$ is also small, we obtain a contradiction to $G$ being a counterexample. Hence $y$ is big.
\end{proof}
\begin{Lemma}\label{aumoins23}
Every 3-vertex is incident to at least two $(3,big)$-edges.
\end{Lemma}
\begin{proof}
Let $(T_1,T_2)$ be a \stf of $G$ and let $x$ be a 3-vertex. Further, let $e=xy$ be the special edge of $x$ with respect to $(T_1,T_2)$. By Proposition \ref{cfzz}, we obtain that $y$ is big. By symmetry, we may suppose that $e \in E(T_1)$.  Now consider a tree-mapping function $\sigma:E(T_1)\rightarrow E(T_2)$. By Proposition \ref{trivial2}, we have $\sigma(e)=f$ for some $f=xz \in \delta_G(x)$. By definition, $(T_1-e+f,T_2-f+e)$ is a \stf of $G$. Further observe that $f$ is the special edge of $x$ with respect to $(T_1-e+f,T_2-f+e)$. Now Proposition \ref{cfzz} yields that $z$ is big. It follows that both $e$ and $f$ are $(3,big)$-edges incident to $x$.
\end{proof}
\subsubsection{New preliminary results}\label{tech1}
In order to give more structural results, we need to make some finer distinctions between the 3-vertices. We call a 3-vertex {\it rich} if it is incident to three $(3,big)$-edges and {\it poor} if it is incident to exactly two $(3,big)$-edges. Observe that every 3-vertex is either rich or poor by Lemma \ref{aumoins23}. Further, we call an 8-vertex {\it critical} if it is incident to one $(2,8)$-edge and seven $(3,8)$-edges. We now give two slightly technical lemmas. Using recursive arguments, we show that certain degenerate cases cannot occur in the neighborhood of a critical 8-vertex. The first lemma shows that no critical 8-vertex can be linked to a poor 3-vertex by two parallel edges. 
\begin{Lemma}\label{drzrz}
Let $x$ be a critical 8-vertex and $y$ a poor 3-vertex. Then $G$ does not contain two parallel edges between $x$ and $y$.
\end{Lemma}
\begin{proof}
Suppose otherwise, let $e,f$ be the two edges between $x$ and $y$ and let $yz$ be the unique edge incident to $y$ such that $z \neq x$. By symmetry, we may suppose that $yz \in E(T_1)$. Now let $G'=G-y+xz, T_1'=T_1-y+xz$ and $T_2'=T_2-y$.  We obtain that $(T_1',T_2')$ is a \stf of $G'$, so $G'$ is a double tree. As $G'$ is smaller than $G$, we obtain that $G'$ has a \stf $(S_1',S_2')$ such that $|d_{S'_1}(v)-d_{S'_2}(v)|\leq 5$ for all $v \in V(G')$. We now prove that we can impose a little extra property on this \stf of $G$.
\begin{Claim}\label{trdft}
There is a \stf $(S_1'',S_2'')$ of $G'$ such that $xz \in E(S''_1), d_{S''_1}(x)\geq 2$ and $|d_{S''_1}(v)-d_{S''_2}(v)|\leq 5$ for all $v \in V(G')$.
\end{Claim}
\begin{proof}[Proof of Claim]
 By symmetry, we may suppose that $xz \in E(S_1')$. If $d_{S_1'}(x)\geq 2$, there is nothing to prove. If $d_{S'_1}(x)=1$, let $\sigma:E(S'_1)\rightarrow E(S'_2)$ be a tree-mapping function. By Proposition \ref{trivial2}, we have $\sigma(xz) \in \delta_{G'}(x)$. Let $S_1''=S_2'-\sigma(xz)+xz$ and  $S_2''=S_1'-xz+\sigma(xz)$.

 By definition, $(S_1'',S_2'')$ is a \stf of $G'$. By construction, we have $xz \in E(S_1'')$. Next, we have  $d_{S_1''}(x)=d_{S_2'}(x)=d_{G'}(x)-d_{S_1'}(x)=d_{G}(x)-1-d_{S_1'}(x)=8-1-1=6\geq 2$. Further, we have $|d_{S''_1}(v)-d_{S''_2}(v)|=|d_{S'_1}(v)-d_{S'_2}(v)|\leq 5$ for all $v \in V(G')-(N_{G'}(x) \cup z)$. As $y$ is poor and there are two edges linking $x$ and $y$ in $G$, we obtain that $z$ is small. As $x$ is critical, all vertices in $N_G(x)$ are also small. Hence we are done by Proposition \ref{trivial}.
\renewcommand{\qedsymbol}{$\blacksquare$}
\end{proof}

Now let $S_1=S_1''-xz+\{e,yz\}$ and $S_2=S_2''+f$.

 Clearly, we have that $(S_1,S_2)$ is a \stf of $G$. We have $|d_{S_1}(v)-d_{S_2}(v)|=|d_{S''_1}(v)-d_{S''_2}(v)|\leq 5$ for all $v \in V(G)-\{x,y\}$. Observe that by Claim \ref{trdft}, we have $d_{S_1}(x)=d_{S_1''}(x)\geq 2$, so $d_{S_2}(x)= d_G(x)-d_{S_1}(x)\leq 8-2=6$ and by construction, $d_{S_2}(x)= d_{S''_2}(x)+1\geq 2$. This yields $|d_{S_1}(x)-d_{S_2}(x)|=|d_{G}(x)-2d_{S_2}(x)| \leq 4 < 5$. As $y$ is small, by Proposition \ref{trivial}, we obtain a contradiction to $G$ being a counterexample.
\end{proof}
The next lemma shows that the 3-vertices in the neighborhood of a critical 8-vertex cannot be linked by edges. 
\begin{Lemma}\label{notadj}
Let $x$ be a critical 8-vertex and $xy_1,xy_2$ be two $(3,8)$-edges incident to $x$ with $y_1 \neq y_2$. Then $E(G)$ does not contain an edge between $y_1$ and $y_2$. 
\end{Lemma}
\begin{proof}
Suppose otherwise. By Lemma \ref{aumoins23}, there is exactly one edge between $y_1$ and $y_2$. Further, by Lemma \ref{drzrz}, for $i=1,2$, we obtain that $y_i$ is incident to an edge $y_iz_i$ with $z_i \notin \{x,y_{3-i}\}$, possibly $z_1=z_2$. We now prove through two claims that $G$ has a \stf which satisfies some extra properties. This allows for a recursive argument then.
\begin{Claim}\label{vgu}
There is a \stf $(U_1,U_2)$ of $G$ which satisfies $|\{xy_1,xy_2,y_1z_1,y_2z_2\}\cap E(U_1)|=2$.
\end{Claim}
\begin{proof}[Proof of Claim]
As $G$ is a double tree, there is \stf $(R_1,R_2)$ of $G$.
By symmetry, we may suppose that $|\{xy_1,xy_2,y_1z_1,y_2z_2\}\cap E(R_1)| \geq |\{xy_1,xy_2,y_1z_1,y_2z_2\}\cap E(R_2)|$. If $|\{xy_1,xy_2,y_1z_1,y_2z_2\}\cap E(R_1)|=2$, there is nothing to prove. If $|\{xy_1,xy_2,y_1z_1,y_2z_2\}\cap E(R_1)|=4$, then there is no edge leaving $\{y_1,y_2\}$ in $R_2$, a contradiction to $(R_1,R_2)$ being a \stf of $G$. We may hence suppose that $|\{xy_1,xy_2,y_1z_1,y_2z_2\}\cap E(R_1)|=3$. By symmetry, we may further suppose that $xy_1,y_1z_1\in E(R_1)$. As $(R_1,R_2)$ is a \stf of $G$, this yields that $y_1y_2 \in E(R_2)$. Now consider a tree-mapping function $\sigma:E(R_2)\rightarrow E(R_1)$. By Proposition \ref{trivial2}, we obtain that $\sigma(y_1y_2)\in \{xy_1,y_1z_1\}$. Let $(U_1,U_2)=(R_1-\sigma(y_1y_2)+y_1y_2,R_2-y_1y_2+\sigma(y_1y_2))$. Now $(U_1,U_2)$ is a \stf of $G$ with $|\{xy_1,xy_2,y_1z_1,y_2z_2\}\cap E(U_1)|=|\{xy_1,xy_2,y_1z_1,y_2z_2\}\cap E(R_1)|-1=2$.
\renewcommand{\qedsymbol}{$\blacksquare$}
\end{proof}

\begin{Claim}\label{rtg6zhu}
There is a \stf $(T_1,T_2)$ of $G$ which satisfies $xy_1,y_2z_2\in E(T_1)$ and $xy_2,y_1z_1\in E(T_2)$.
\end{Claim}
\begin{proof}[Proof of Claim]
Let $(U_1,U_2)$ like in Claim \ref{vgu}. By symmetry, we may suppose that $xy_1 \in E(U_1)$. By Claim \ref{vgu}, exactly one of the three edges $y_1z_1,xy_2$ and $y_2z_2$ is in $E(U_1)$. If $y_2z_2 \in E(U_1)$, there is nothing to prove. It hence suffices to consider the two remaining cases.
\begin{case}
$y_1z_1 \in E(U_1)$
\end{case}
\begin{proof}[Proof of Case]
As $U_2$ is a spanning tree and $y_1$ is a 3-vertex, we obtain that $y_1y_2 \in E(U_2)$. By Claim \ref{vgu}, we have $xy_2,y_2z_2 \in E(U_2)$. It follows that no edge of $E(U_1)$ is incident to $y_2$. This contradicts $U_1$ being a spanning tree of $G$.
\renewcommand{\qedsymbol}{$\blacksquare$}
\end{proof}
\begin{case}
$xy_2 \in E(U_1)$
\end{case}
\begin{proof}[Proof of Case]
Observe that $y_1y_2 \in E(U_2)$ as $U_1$ cannot contain the triangle $xy_1y_2$. Let $\sigma:E(U_1)\rightarrow E(U_2)$ be a tree-mapping function from $U_1$ to $U_2$. By Proposition \ref{triangle}, we obtain that either $\sigma(xy_1)\neq y_1y_2$ or $\sigma(xy_2)\neq y_1y_2$.

 First assume that $\sigma(xy_2)\neq y_1y_2$. By Proposition \ref{trivial2}, we obtain that $\sigma(xy_2)\in \delta_G(y_2)$, so  $\sigma(xy_2) = y_2z_2$. Let $(T_1,T_2)=(U_1-xy_2+y_2z_2,U_2-y_2z_2+xy_2)$. Now $(T_1,T_2)$ is a \stf of $G$ with $xy_1,y_2z_2\in E(T_1)$ and $xy_2,y_1z_1\in E(T_2)$.

 Now assume that $\sigma(xy_1)\neq y_1y_2$. By Proposition \ref{trivial2}, we obtain that $\sigma(xy_1)\in \delta_G(y_1)$, so  $\sigma(xy_1) = y_1z_1$. Let $(T_1,T_2)=(U_2-y_1z_1+xy_1,U_1-xy_1+y_1z_1)$. Now $(T_1,T_2)$ is a \stf of $G$ with $xy_1,y_2z_2\in E(T_1)$ and $xy_2,y_1z_1\in E(T_2)$.
\renewcommand{\qedsymbol}{$\blacksquare$}
\end{proof}
Now the case distinction is complete and so the proof of Claim \ref{rtg6zhu} is finished.
\renewcommand{\qedsymbol}{$\blacksquare$}
\end{proof}

Let $(T_1,T_2)$ like in Claim \ref{rtg6zhu}. By symmetry, we may suppose that $y_1y_2 \in E(T_1)$. Let $G'$ be obtained from $G$ by deleting $y_1$ and $y_2$ and creating a new vertex $y$, two new edges $e,f$ linking $x$ and $y$ and one edge linking $y$ and $z_2$. Further, let $ T'_1=T_1-\{y_1,y_2\}+\{e,yz_2\}$ and $T_2'=T_2-\{y_1,y_2\}+f$.  We obtain that $(T_1',T_2')$ is a \stf of $G'$, so $G'$ is a double tree. As $G'$ is smaller than $G$, we obtain that $G'$ has a \stf $(S_1',S_2')$ such that $|d_{S'_1}(v)-d_{S'_2}(v)|\leq 5$ for all $v \in V(G')$. As $S_1'$ and $S_2'$ are spanning trees and by symmetry, we may suppose that $e,yz_2 \in E(S'_1)$ and $f \in E(S'_2)$.

 Recall that $d_{S_2'}(z_1)-d_{S_1'}(z_1) \leq 5$. We now show how to obtain a \stf of $G$ of the desired form in two different cases which reflect whether the bound in this inequality is attained or not.
\setcounter{case}{0}
\begin{case}
$d_{S_2'}(z_1)-d_{S_1'}(z_1)<5$.
\end{case}
\begin{proof}[Proof of Case]
Let $S_1=S'_1-y+\{xy_1,y_1y_2,y_2z_2\}$ and $S_2=S'_2-y+\{xy_2,y_1z_1\}$.  Clearly, $(S_1,S_2)$ is a \stf of $G$. We have $|d_{S_1}(v)-d_{S_2}(v)|=|d_{S'_1}(v)-d_{S'_2}(v)|\leq 5$ for all $v \in V(G)-\{y_1,y_2,z_1\}$. Further, by construction, we have $d_{S_1}(z_1)-d_{S_2}(z_1)=d_{S_1'}(z_1)-d_{S_2'}(z_1)-1<5$ and by the case distinction, we have $d_{S_2}(z_1)-d_{S_1}(z_1)=d_{S_2'}(z_1)-d_{S_1'}(z_1)+1\leq 5$. This yields $|d_{S_1}(z_1)-d_{S_2}(z_1)|\leq 5$. As $y_1$ and $y_2$ are small, by Proposition \ref{trivial}, we obtain a contradiction to $G$ being a counterexample.
\renewcommand{\qedsymbol}{$\blacksquare$}
\end{proof}
\begin{case}
$d_{S_2'}(z_1)-d_{S_1'}(z_1)=5$. 
\end{case}
\begin{proof}[Proof of Case]
Let $S_1=S'_1-y+\{xy_2,y_1z_1,y_2z_2\}$ and $S_2=S'_2-y+\{xy_1,y_1y_2\}$.  Clearly, $(S_1,S_2)$ is a \stf of $G$. We have $|d_{S_1}(v)-d_{S_2}(v)|=|d_{S'_1}(v)-d_{S'_2}(v)|\leq 5$ for all $v \in V(G)-\{y_1,y_2,z_1\}$. Further, by construction and the case distinction, we have $d_{S_2}(z_1)-d_{S_1}(z_1)=d_{S_2'}(z_1)-d_{S_1'}(z_1)-1=4$, so $|d_{S_1}(z_1)-d_{S_2}(z_1)|\leq 5$. As $y_1$ and $y_2$ are small, by Proposition \ref{trivial}, we obtain a contradiction to $G$ being a counterexample.
\renewcommand{\qedsymbol}{$\blacksquare$}
\end{proof}
Now the case distinction is complete and so the proof of Lemma \ref{notadj} is finished.
\end{proof}
\subsubsection{New main lemma}\label{hstrrh}
After having excluded some degenerate cases in Section \ref{tech1}, we are now ready to provide the new main lemma for the improvement of the constant in Theorem \ref{main} from 6 to 5.
\begin{Lemma}\label{new}
Let $x$ be a critical 8-vertex and $(T_1,T_2)$ a \stf of $G$. Then $x$ is incident to at least two $(3,8)$-edges $xy_1,xy_2$ such that $y_1$ and $y_2$ are rich and $xy_i$ is not the special edge of $y_i$ with respect to $(T_1,T_2)$ for $i=1,2$.
\end{Lemma}
\begin{proof}
Suppose otherwise. Then there is a spanning tree factorization $(T_1,T_2)$ of $G$ for which $x$ is incident to at most one $(3,8)$-edge $xy$ such that $y$ is rich and $xy$ is not the special edge of $y$ with respect to $(T_1,T_2)$. Among all such spanning tree factorizations, choose $(T_1^*,T_2^*)$ in a way that the number of $(3,8)$-edges $xy$ such that $y$ is poor and $xy$ is not the special edge of $y$ with respect to $(T_1^*,T_2^*)$ is minimized.
\begin{Claim}\label{sert}
For all $(3,8)$-edges $xy$ such that $y$ is poor we have that $xy$ is the special edge of $y$ with respect to $(T_1^*,T_2^*)$.
\end{Claim}
\begin{proof}[Proof of Claim]
Suppose otherwise, so there is at least one edge $xz$ such that $z$ is poor and $xz$ is not the special edge of $z$ with respect to $(T_1^*,T_2^*)$. By symmetry, we may suppose that $xz \in E(T_1^*)$. As $xz$ is not the special edge of $z$ and $z$ is a 3-vertex, we obtain that $z$ is incident to two more edges $zz_1\in E(T_1^*)$ and $zz_2 \in E(T_2^*)$. By definition, $zz_2$ is the special edge of $z$ with respect to $(T_1^*,T_2^*)$. We obtain by Proposition \ref{cfzz} that $z_2$ is big. Let $\sigma:E(T_2^*)\rightarrow E(T_1^*)$ be a tree-mapping function.  Let $U_1=T^*_1-\sigma(zz_2)+zz_2$ and $U_2=T_2^*-zz_2+\sigma(zz_2)$. By Proposition \ref{trivial2}, we obtain that $\sigma(zz_2)\in xz,zz_1$. If $\sigma(zz_2)=zz_1$, then $zz_1$ is the special edge of $z$ with repect to $(U_1,U_2)$. It follows from Proposition \ref{cfzz} that $z_1$ is big. As $x$ is big by assumption, we hence obtain that all of $x,z_1$ and $z_2$ are big. This contradicts $z$ being poor.

We hence have $\sigma(zz_2)=xz$, so $xz$ is the special edge of $z$ with respect to $(U_1,U_2)$. Further, by Lemma \ref{notadj}, we obtain that $\delta_G(y)\cap\delta_G(z)= \emptyset$ for all $y \in N_G(x)-z$. It follows that for all $xy\in \delta_G(x)$ with $y \neq z$, we have that $xy$ is the special edge of $y$ with respect to $(U_1,U_2)$ if and only if  $xy$ is the special edge of $y$ with respect to $(T^*_1,T^*_2)$. As there is only one edge between $x$ and $z$ by Lemma \ref{drzrz}, the number of $(3,8)$-edges $xy$ such that $y$ is poor and $xy$ is not the special edge of $y$ with respect to $(U_1,U_2)$ is strictly smaller than the number of $(3,8)$-edges $xy$ such that $y$ is poor and $xy$ is not the special edge of $y$ with respect to $(T^*_1,T^*_2)$. Further, the number of $(3,8)$-edges $xy$ such that $y$ is rich and $xy$ is not the special edge of $y$ with respect to $(U_1,U_2)$ is the same as the number of $(3,8)$-edges $xy$ such that $y$ is rich and $xy$ is not the special edge of $y$ with respect to $(T^*_1,T^*_2)$. We hence obtain a contradiction to the choice of $(T_1^*,T_2^*)$.
\renewcommand{\qedsymbol}{$\blacksquare$}
\end{proof}
By symmetry and the assumption, we may suppose that there is no $(3,8)$-edge $xy\in E(T_1^*)$ such that $y$ is rich and $xy$ is not the special edge of $y$ with respect to $(T_1^*,T_2^*)$. Let $X=x \cup N_{T_1^*}(x)$. By definition, we obtain $\delta_{T_1^*}(X)\cap \delta_G(x)=\emptyset$. For all rich 3-vertices $y \in X$, we obtain $\delta_{T_1^*}(X)\cap \delta_G(y)=\emptyset$ by assumption. For all poor 3-vertices $y \in X$, we obtain $\delta_{T_1^*}(X)\cap \delta_G(y)=\emptyset$ by Claim \ref{sert}.

 As $x$ is critical, there is a unique 2-vertex $z$ in $N_G(x)$. If $z \notin X$, we obtain $\delta_{T_1^*}(X)=\emptyset$ and $z \in V(G)-X$, so $V(G)-X \neq \emptyset$. This contradicts $T_1^*$ being a spanning tree. If $z \in X$, then, as $T_2^*$ is a spanning tree, we obtain $\delta_{T_1^*}(X)\cap \delta_G(z)=\emptyset$ and so again $\delta_{T_1^*}(X)=\emptyset$. Further, by Lemmas \ref{2gross} and \ref{only}, there is a big vertex $z'\in N_G(z)-x$. As $x$ is critical, all neighbors of $x$ in $G$ are small and so $z' \in V(G)-X$. Again, we obtain $V(G)-X \neq \emptyset$, a contradiction to $T_1^*$ being a spanning tree.
\end{proof}
\subsection{The discharging procedure}\label{charge}

We are now ready to finish the proof of Theorem \ref{main} by a discharging procedure which uses the structural results obtained in Section \ref{struc}.

The initial charge of every vertex $v \in V(G)$ is $d_G(v)$. We now describe the discharging rule. Let $(T_1,T_2)$ be a \stf  of $G$. Every big vertex $x$ sends the following charge along every $(2,big)$-edge or $(3,big)$-edge $e=xy$ it is incident to:
\begin{itemize}
\item 1 if $y$ is a 2-vertex,
\item $\frac{1}{2}$ if $y$ is a poor 3-vertex,
\item $\frac{1}{2}$ if $y$ is a rich 3-vertex and $e$ is the special edge of $y$ with respect to $(T_1,T_2)$,
\item $\frac{1}{4}$ if $y$ is a rich 3-vertex and $e$ is not the special edge of $y$ with respect to $(T_1,T_2)$.
\end{itemize}
For every $v \in V(G)$, we denote by $c_f(v)$ the final charge of $v$ after the discharging procedure.
\begin{Claim}
$c_f(v)\geq 4$ for all $v \in V(G)$.
\end{Claim}
\begin{proof}[Proof of Claim]
Let $v \in V(G)$. 


If $v$ is a 2-vertex, then by Lemma \ref{2gross}, we obtain that $v$ is incident to two $(2,big)$-edges. Hence $v$ receives a charge of 1 along both its incident edges and does not send any charge. This yields $c_f(v)=d_G(v)+1+1=4$.

If $v$ is a 3-vertex, by Lemma \ref{aumoins23}, we have that $v$ is either poor or rich. If $v$ is poor, then $v$ receives a charge of $\frac{1}{2}$ along both the $(3,big)$-edges $v$ is incident to and does not send any charge. This yields $c_f(v)=d_G(v)+\frac{1}{2}+\frac{1}{2}=4$.  If $v$ is rich, then $v$ receives a charge of $\frac{1}{2}$ along its special edge and a charge of $\frac{1}{4}$ along the other two edges  $v$ is incident to and does not send any charge. This yields $c_f(v)=d_G(v)+\frac{1}{2}+\frac{1}{4}+\frac{1}{4}=4$.

If $4 \leq d_G(v) \leq 7$, then $v$ neither sends nor receives any charge, so $c_f(v)=d_G(v)\geq 4$.

Now let $v$ be an 8-vertex. If $v$ is critical, then by Lemma \ref{new}, $v$ is incident to at least two $(3,8)$-edges $xy_1,xy_2$ such that $y_1$ and $y_2$ are rich and $xy_i$ is not the special edge of $y_i$ with respect to $(T_1,T_2)$ for $i=1,2$. We obtain that $v$ sends a charge of $\frac{1}{4}$ along both these edges. Further, we have that $v$ sends a charge of $\frac{1}{2}$ along the remaining $(3,8)$-edges $v$ is incident to and $1$ along the $(2,8)$-edge $v$ is incident to. This yields $c_f(v)\geq d_G(v)-2 \cdot \frac{1}{4}-5 \cdot \frac{1}{2}-1=4$.

 Otherwise, by Lemma \ref{only}, either $v$ is not incident to a $(2,8)$-edge or $v$ is incident to exactly one $(2,8)$-edge and at most six $(3,8)$-edges. In the former case, we obtain that $v$ sends a charge of at most $\frac{1}{2}$ along all the edges $v$ is incident to yielding $c_f(v)\geq d_G(v)-8\cdot \frac{1}{2}=4$.

  In the latter case, we obtain that $v$ sends a charge of 1 along one edge, no charge along along at least one edge and a charge of at most $\frac{1}{2}$ along the remaining edges yielding $c_f(v)\geq d_G(v)-1-6\cdot \frac{1}{2}=4$. 

If $d_G(v)\geq 9$, then by Lemma \ref{only}, $v$ is incident to at most one $(2,big)$-edge. It follows that $v$ sends a charge of 1 along at most one edge and a charge of at most $\frac{1}{2}$ along all other edges $v$ is incident to. This yields $c_f(v)\geq d_G(v)-1-\frac{1}{2}(d_G(v)-1)=\frac{1}{2}d_G(v)-\frac{1}{2}\geq 4$.
\renewcommand{\qedsymbol}{$\blacksquare$}
\end{proof}

As the total final charge is the same as the total initial charge, we obtain
\begin{equation*}
|E(G)|=\frac{1}{2}\sum_{v \in V(G)}d_G(v)=\frac{1}{2}\sum_{v \in V(G)}c_f(v)\geq\frac{1}{2}\sum_{v \in V(G)}4=2|V(G)|,
\end{equation*}

a contradiction to Proposition \ref{aretes}. This finishes the proof of Theorem \ref{main}.

\section{The general case}\label{neu}

This section is dedicated to proving Theorem \ref{newmain}. We need one more definition which is useful throughout the proof.
For some nonnegative integer $k$, a constant $c_k \in \mathbb{R}$ is called {\it $k$-feasible} if every $k$-multiple tree has a \stf $(T_1,\ldots,T_k)$ such that $|d_{T_i}(v)-\frac{d_G(v)}{k}|\leq c_k$ for all $v \in V(G)$ and $i \in \{1,\ldots,k\}$. We will show later how the existence of an appropriate $k$-feasible constant for every positive integer $k$ easily implies Theorem \ref{newmain}.

We first give some results that show how to recursively obtain $k$-feasible constants. We then combine these results to prove Theorem \ref{newmain}.

The following result shows that feasible constants are available for small values of $k$.
\begin{Lemma}\label{rfde}
The constant $0$ is 1-feasible and $\frac{5}{2}$ is 2-feasible.
\end{Lemma}
\begin{proof}
The first part follows immediately from the fact that a 1-multiple tree is a tree. For the second part, let $G$ be a double tree. By Theorem \ref{main}, there exists a \stf $(T_1,T_2)$ of $G$ such that $|d_{T_1}(v)-d_{T_2}(v)|\leq 5$ for all $v \in V(G)$. For every $i \in \{1,2\}$ and $v \in V(G)$, we obtain $|d_{T_i}(v)-\frac{d_G(v)}{2}|=|d_{T_i}(v)-\frac{d_{T_i}(v)+d_{T_{3-i}}(v)}{2}|=|\frac{d_{T_i}(v)-d_{T_{3-i}}(v)}{2}|\leq \frac{5}{2}$.
\end{proof}
We next show how to conclude the existence of a $2k$-feasible constant from the existence of a $k$-feasible constant. While this result is not strictly necessary to solve Conjecture \ref{ddfd}, we need it to obtain the logarithmic bound in Theorem \ref{newmain}.
\begin{Lemma}\label{uhhu}
Suppose that for some positive integer $k$, there is a $k$-feasible constant $c_k$. Then $c_k+\frac{5}{2}$ is a $2k$-feasible constant.
\end{Lemma}
\begin{proof}
Let $G$ be a $2k$-multiple tree and $(U_1,\ldots,U_{2k})$ a \stf of $G$. Let $H_1=(V(G), E(U_1)\cup \ldots \cup E(U_k))$ and $H_2=(V(G), E(U_{k+1})\cup \ldots \cup E(U_{2k}))$. Clearly, both $H_1$ and $H_2$ are $k$-multiple trees. Hence, by assumption, there exist a spanning tree factorization $(U_1',\ldots,U_k')$ of $H_1$ and a \stf $(U_{k+1}',\ldots,U_{2k}')$ of $H_2$ such that $|d_{U'_i}(v)-\frac{d_{H_1}(v)}{k}| \leq c_k$ for all $v \in V(G)$ and $i \in \{1,\ldots,k\}$ and $|d_{U'_i}(v)-\frac{d_{H_2}(v)}{k}| \leq c_k$ for all $v \in V(G)$ and $i \in \{k+1,\ldots,2k\}$. Now, for $i \in \{1,\ldots,k\}$, let $F_i=(V(G),E(U_i')\cup E(U_{k+i}'))$. Clearly, $F_i$ is a double tree for all $i \in \{1,\ldots,k\}$. By Lemma \ref{rfde}, there is a spanning tree factorization $(T_i,T_{k+i})$ of $F_i$ such that $|d_{T_i}(v)-\frac{d_{F_i}(v)}{2}| \leq \frac{5}{2}$ and $|d_{T_{k+i}}(v)-\frac{d_{F_i}(v)}{2}| \leq \frac{5}{2}$ for all  $v \in V(G)$ and $i \in \{1,\ldots,k\}$. For all $v \in V(G)$ and $i \in \{1,\ldots,k\}$, we obtain 
\begin{align*}|d_{T_i}(v)-\frac{d_{G}(v)}{2k}| &= |d_{T_i}(v)-\frac{d_{F_i}(v)}{2}+\frac{1}{2}(d_{F_i}(v)-\frac{d_{G}(v)}{k})|\\
&\leq |d_{T_i}(v)-\frac{d_{F_i}(v)}{2}|+\frac{1}{2}|d_{F_i}(v)-\frac{d_{G}(v)}{k}|\\
&\leq \frac{5}{2}+\frac{1}{2}|d_{U'_i}(v)+d_{U'_{k+i}}(v)-\frac{d_{H_1}(v)+d_{H_2}(v)}{k}|\\
&\leq \frac{5}{2}+\frac{1}{2}(|d_{U'_i}(v)-\frac{d_{H_1}(v)}{k}|+|d_{U'_{k+i}}(v)-\frac{d_{H_2}(v)}{k}|)\\
&\leq \frac{5}{2}+\frac{1}{2}(c_k+c_k)\\
&=c_k+\frac{5}{2}.
\end{align*}

Similarly, we have $|d_{T_i}(v)-\frac{d_{G}(v)}{2k}|\leq c_k+\frac{5}{2}$ for all $i \in \{k+1,\ldots,2k\}$. It follows that $(T_1,\ldots,T_k)$ has the desired properties. This finishes the proof.
\end{proof}
We now wish to obtain a second result which allows to get a $(k+1)$-feasible constant from a $k$-feasible constant. In order to prove this result, we first need the following lemma whose proof is pretty technical.

\begin{Lemma}\label{baum}
Suppose that for some positive integer $k$, there exists a $k$-feasible constant $c_k$. Let $G$ be a $(k+1)$-multiple tree. Then there is a \stf $(T_1,\ldots,T_{k+1})$ of $G$ such that $|d_{T_1}(v)-\frac{d_G(v)}{k+1}|\leq \frac{k(c_k+5)}{k+1}$ for all $v \in V(G)$.
\end{Lemma}
\begin{proof}
We recursively define a series $(T_1^{j},\ldots,T_{k+1}^{j})_{j \in \mathbb{N}}$ of spanning tree factorizations of $G$. First let $(T_1^{0},\ldots,T_{k+1}^{0})$ be an arbitrary \stf of $G$. Now suppose that we have already created $(T_1^{j},\ldots,T_{k+1}^{j})$ and want to create $(T_1^{j+1},\ldots,T_{k+1}^{j+1})$ for some $j \in \mathbb{N}$. First let $H^{j}=G-E(T_1^{j})$. Observe that $H^{j}$ is a $k$-multiple tree. Hence by assumption there is a spanning tree factorization $(U_2^{j},U_3^{j},\ldots,U_{k+1}^{j})$ of $H^{j}$ which satisfies $|d_{U_i^{j}}(v)-\frac{d_{H^{j}}(v)}{k}|\leq c_k$ for all $v \in V(G)$ and $i \in \{2,\ldots,k+1\}$. Now consider $F^{j}=(V(G),E(T_1^{j})\cup E(U_2^{j}))$. Observe that $F^{j}$ is a double tree. By Lemma \ref{rfde}, there is a \stf $(R_1^{j},R_2^{j})$ of $F^{j}$ such that $|d_{R_i^{j}}(v)-\frac{d_{F^{j}}(v)}{2}|\leq \frac{5}{2}$ for all $v \in V(G)$ and $i \in \{1,2\}$. Let $(T_1^{j+1},\ldots,T_{k+1}^{j+1})=(R_1^{j},R_2^{j},U_3^{j},\ldots,U_{k+1}^{j})$.
\medskip

We show in the following that for sufficiently large $j$, the \stf $(T_1^{j},\ldots,T_{k+1}^{j})$ has the desired properties.

For every $v \in V(G)$, we define a series $(x_v^{j})_{j \in \mathbb{N}}$ by $x_v^0=0$ and $x_v^{j+1}=x_v^j(\frac{1}{2}-\frac{1}{2k})+\frac{d_G(v)}{2k}-(\frac{c_k}{2}+\frac{5}{2})$ for all $j \geq 0$ and a series $(y_v^{j})_{j \in \mathbb{N}}$ by $y_v^0=d_G(v)$ and $y_v^{j+1}=y_v^j(\frac{1}{2}-\frac{1}{2k})+\frac{d_G(v)}{2k}+(\frac{c_k}{2}+\frac{5}{2})$ for all $j \geq 0$.

\begin{Claim}\label{gzuzugzu}
For all $v \in V(G)$ and $j \in \mathbb{N}$, we have $x_v^{j}\leq d_{T_1^{j}}(v)\leq y_v^{j}$.
\end{Claim}
\begin{proof}[Proof of Claim]
The statement is evident for $j=0$. Now suppose that the statement holds for all integers up to some $j \in \mathbb{N}$. By the definition of the spanning trees and $(\frac{1}{2}-\frac{1}{2k})\geq 0$, we obtain

\begin{align*}
d_{T_1^{j+1}}(v)&\geq \frac{d_{F^{j}}(v)}{2}-\frac{5}{2}\\
&=\frac{1}{2}(d_{T_1^{j}}(v)+d_{U_2^{j}}(v))-\frac{5}{2}\\
&\geq \frac{1}{2}(d_{T_1^{j}}(v)+\frac{d_{H^{j}}(v)}{k}-c_k)-\frac{5}{2}\\
&= \frac{1}{2}(d_{T_1^{j}}(v)+\frac{d_G(v)-d_{T_1^{j}}(v)}{k}-c_k)-\frac{5}{2}\\
&=(\frac{1}{2}-\frac{1}{2k})d_{T_1^{j}}(v)+\frac{d_G(v)}{2k}-\frac{c_k}{2}-\frac{5}{2}\\
&\geq (\frac{1}{2}-\frac{1}{2k})x_v^{j}+\frac{d_G(v)}{2k}-\frac{c_k}{2}-\frac{5}{2}\\
&=x_v^{j+1}
\end{align*}
 and similarly 

\begin{align*}
d_{T_1^{j+1}}(v)&\leq \frac{d_{F^{j}}(v)}{2}+\frac{5}{2}\\
&=\frac{1}{2}(d_{T_1^{j}}(v)+d_{U_2^{j}}(v))+\frac{5}{2}\\
&\leq \frac{1}{2}(d_{T_1^{j}}(v)+\frac{d_{H^{j}}(v)}{k}+c_k)+\frac{5}{2}\\
&= \frac{1}{2}(d_{T_1^{j}}(v)+\frac{d_G(v)-d_{T_1^{j}}(v)}{k}+c_k)+\frac{5}{2}\\
&=(\frac{1}{2}-\frac{1}{2k})d_{T_1^{j}}(v)+\frac{d_G(v)}{2k}+\frac{c_k}{2}+\frac{5}{2}\\
&\leq (\frac{1}{2}-\frac{1}{2k})y_v^{j}+\frac{d_G(v)}{2k}+\frac{c_k}{2}+\frac{5}{2}\\
&=y_v^{j+1}.
\end{align*}
\renewcommand{\qedsymbol}{$\blacksquare$}
\end{proof}
For every $v \in V(G)$, we now define $x_v=\frac{d_G(v)}{k+1}-\frac{k(c_k+5)}{k+1}$ and $y_v=\frac{d_G(v)}{k+1}+\frac{k(c_k+5)}{k+1}$.

\begin{Claim}
For all $v \in V(G)$, there is some $n_v \in \mathbb{N}$ such that $x_v \leq d_{T_1^{j}}(v)\leq y_v$ for all $j \geq n_v$.
\end{Claim}
\begin{proof}[Proof of Claim]
Let $v \in V(G)$. By Proposition \ref{gzigzi}, we have 
\begin{align*}
\lim_{j \rightarrow \infty}x_v^{j}&=\frac{\frac{d_G(v)}{2k}-(\frac{c_k}{2}+\frac{5}{2})}{1-(\frac{1}{2}-\frac{1}{2k})}\\
&=\frac{\frac{d_G(v)}{2k}-(\frac{c_k}{2}+\frac{5}{2})}{\frac{k+1}{2k}}\\
&=\frac{d_G(v)}{k+1}-\frac{k(c_k+5)}{k+1}\\
&=x_v
\end{align*}

and

\begin{align*}
\lim_{j \rightarrow \infty}y_v^{j}&=\frac{\frac{d_G(v)}{2k}+(\frac{c_k}{2}+\frac{5}{2})}{1-(\frac{1}{2}-\frac{1}{2k})}\\
&=\frac{\frac{d_G(v)}{2k}+(\frac{c_k}{2}+\frac{5}{2})}{\frac{k+1}{2k}}\\
&=\frac{d_G(v)}{k+1}+\frac{k(c_k+5)}{k+1}\\
&=y_v.
\end{align*}
By Claim \ref{gzuzugzu}, we have $d_{T_1^{j}}(v)\geq x_v^{j}$ for all $j \in \mathbb{N}$. Further observe that $d_{T_1^{j}}(v)\in \{0,\ldots,d_G(v)\}$  for all $j \in \mathbb{N}$. Hence by Proposition \ref{ddaqd}, we obtain that there is some $n_0 \in \mathbb{N}$ such that $d_{T_1^{j}}(v)\geq x_v$ for all $j \geq n_0$. 

Further, by Claim \ref{gzuzugzu}, we have $d_{T_1^{j}}(v)\leq y_v^{j}$ for all $j \in \mathbb{N}$. By applying Proposition \ref{ddaqd} to $(-y_v^{j})_{j \in \mathbb{N}},-y_v$ and $(-d_{T_1^{j}}(v))_{j \in \mathbb{N}}$, we obtain that there is some $n_1 \in \mathbb{N}$ such that $d_{T_1^{j}}(v)\leq y_v$ for all $j \geq n_1$. 

Choosing $n_v=\max\{n_0,n_1\}$ finishes the proof.
\renewcommand{\qedsymbol}{$\blacksquare$}
\end{proof}
Now let $n^*=\max_{v \in V(G)}n_v$ and $(T_1,\ldots,T_{k+1})=(T_1^{n^*},\ldots,T_{k+1}^{n^*})$. Clearly, $(T_1,\ldots,T_{k+1})$ is a \stf of $G$. Further, for every $v \in V(G)$, we have $d_{T_1}(v)-\frac{d_G(v)}{k+1}\leq y_v-\frac{d_G(v)}{k+1}=\frac{k(c_k+5)}{k+1}$ and $\frac{d_G(v)}{k+1}-d_{T_1}(v)\leq\frac{d_G(v)}{k+1}-x_v=\frac{k(c_k+5)}{k+1}$. This finishes the proof.
\end{proof}
We are now ready to give the second main lemma of this section.
\begin{Lemma}\label{rdztd}
Suppose that for some positive integer $k$, there exists a $k$-feasible constant $c_k$. Then $(1+\frac{1}{k+1})c_k+5$ is a $(k+1)$-feasible constant.
\end{Lemma}
\begin{proof}
Let $G$ be a $(k+1)$-multiple tree. We need to prove that there exists a \stf $(T_1,\ldots,T_{k+1})$ of $G$ such that $|d_{T_i}(v)-\frac{d_G(v)}{k+1}|\leq (1+\frac{1}{k+1})c_k+5$ for all $v \in V(G)$ and $i \in \{1,\ldots,k+1\}$. By Lemma \ref{baum}, there is a \stf $(T_1,U_2,\ldots,U_{k+1})$ of $G$ such that $|d_{T_1}(v)-\frac{d_G(v)}{k+1}|\leq \frac{k(c_k+5)}{k+1}$ for all $v \in V(G)$. Let $H=G-E(T_1)$. Observe that $H$ is a $k$-multiple tree. By assumption, there is a spanning tree factorization $(T_2,\ldots,T_{k+1})$ of $H$ such that $|d_{T_i}(v)-\frac{d_H(v)}{k}|\leq c_k$ for all $v \in V(G)$ and $i \in \{2,\ldots,k+1\}$. We will show that $(T_1,\ldots,T_{k+1})$ is a \stf of $G$ with the desired properties. First observe that $|d_{T_1}(v)-\frac{d_G(v)}{k+1}|\leq \frac{k(c_k+5)}{k+1}\leq (1+\frac{1}{k+1})c_k+5$ for all $v \in V(G)$. Further, for all $v \in V(G)$ and $i \in \{2,\ldots,k+1\}$, we obtain

\begin{align*}
|d_{T_i}(v)-\frac{1}{k+1}d_G(v)|&=|d_{T_i}(v)-\frac{1}{k}d_G(v)+\frac{1}{k(k+1)}d_G(v)|\\
&=|d_{T_i}(v)-\frac{1}{k}d_H(v)-(\frac{1}{k}d_{T_1}(v)-\frac{1}{k(k+1)}d_G(v))|\\
&\leq |d_{T_i}(v)-\frac{1}{k}d_H(v)|+\frac{1}{k}|d_{T_1}(v)-\frac{1}{k+1}d_G(v)|\\
&\leq c_k+\frac{1}{k}\frac{k(c_k+5)}{k+1}\\
&\leq c_k(1+\frac{1}{k+1})+5.
\end{align*}
This finishes the proof.
\end{proof}

We are now ready to prove the following result which will imply Theorem \ref{newmain} after.

\begin{Theorem}\label{tech}
For every positive integer $k$, $11 \log (k)$ is a $k$-feasible constant.
\end{Theorem}
\begin{proof}
The statement holds for $k=1$ by Lemma \ref{rfde}. Now suppose that the statement holds for all integers up to some $k$. We show that it also holds for $k+1$.

First suppose that $k$ is odd, so $k+1=2 \mu$ for some positive integer $\mu$. Recursively, we obtain that $11 \log (\mu)$ is a $\mu$-feasible constant. Now Lemma \ref{uhhu} yields that $11 \log (\mu) +\frac{5}{2}$ is a $(k+1)$-feasible constant. Further, we have $11 \log (\mu) +\frac{5}{2} \leq 11(\log (\mu )+1)=11(\log (\mu) +\log (2))=11\log(2 \mu)=11 \log(k+1)$.

Now suppose that $k$ is even, so $k=2 \mu$ for some positive integer $\mu$. By Lemma \ref{uhhu} and recursively, we obtain that  $11 \log (\mu) +\frac{5}{2}$ is a $k$-feasible constant and hence by Lemma \ref{rdztd}, we obtain that $(1+\frac{1}{2 \mu +1})(11 \log (\mu) +\frac{5}{2})+5$ is a $(k+1)$-feasible constant. Further, by Proposition \ref{loga}, we have

\begin{align*}
(1+\frac{1}{2 \mu +1})(11 \log (\mu) +\frac{5}{2})+5&=11 \log (\mu)+\frac{11 \log (\mu) +\frac{5}{2}}{2 \mu +1}+\frac{5}{2}+5\\
&\leq 11 \log (\mu) + 11\\
& = 11(\log (\mu)+\log (2))\\
&=11\log (2 \mu)\\
&\leq 11 \log (k+1). 
\end{align*}
\end{proof}

We finally show that Theorem \ref{tech} indeed implies Theorem \ref{newmain}.

\begin{proof}(of Theorem \ref{newmain})

Let $G$ be a $k$-multiple tree. By Theorem \ref{tech}, there is a \stf $(T_1,\ldots,T_k)$ of $G$ such that $|d_{T_i}(v)-\frac{d_G(v)}{k}|\leq 11 \log(k)$ for all $v \in V(G)$ and $i \in \{1,\ldots,k\}$. For all $v \in V(G)$ and $i,j \in \{1,\ldots,k\}$, we obtain $|d_{T_i}(v)-d_{T_j}(v)|\leq |d_{T_i}(v)-\frac{d_G(v)}{k}|+|d_{T_j}(v)-\frac{d_G(v)}{k}|\leq 11 \log(k)+11 \log(k)=22 \log(k)$. Hence $(T_1,\ldots, T_k)$ is a \stf with the desired properties.
\end{proof}
\section{Conclusion}\label{conc}

We manage to settle Conjecture \ref{ddfd}. We first give a proof for the case of $k=2$ which also includes an improvement for the constant in Problem \ref{dggfd}. While the result on Conjecture \ref{ddfd} seems to be the much more general one of these contributions, the improvement of the constant is the hardest part of the proof of  Theorem \ref{main}.

After, we give a proof of the general version of Conjecture \ref{ddfd} by applying Theorem \ref{main} repeatedly.

Certainly, this work leaves many open questions. It would be good to know whether the obtained bounds can be improved. Two questions about this improvement seem particularly interesting.
Firstly, one could ask whether the constant in Theorem \ref{main} can be improved. 
\begin{Problem}\label{prob5}
What is the minimum integer $c$ for which every double tree $G$ has a spanning tree factorization $(T_1,T_2)$ such that $|d_{T_1}(v)-d_{T_2}(v)|\leq c$ for all $v \in V(G)$?
\end{Problem}

By Theorem \ref{main}, we have $c \leq 5$ and the construction given by Kriesell in \cite{nk1} shows that $c \geq 2$. Clearly, the value of this constant plus 2 is at least the value of the constant in Problem \ref{feuille}. 

Secondly, we would like to know whether in Theorem \ref{newmain} $c_k$ can be replaced by a global constant.

\begin{Conjecture}
There is an integer $c$ such that every $k$-multiple tree $G$ has a \stf $(T_1,\ldots,T_k)$ such that $|d_{T_i}(v)-d_{T_j}(v)|\leq c$ for all $v \in V(G)$ and $i,j \in \{1,\ldots,k\}$.
\end{Conjecture}

Another question one could ask is of algorithmic nature. We wish to know whether we can decide efficiently if a given double tree has a \stf which is in some way perfectly balanced.

\begin{Problem}\label{prob2}
Can we decide in polynomial time whether a given eulerian double tree $G$ has a \stf $(T_1,T_2)$ such that $d_{T_1}(v)=d_{T_2}(v)$ for all $v \in V(G)$?
\end{Problem}

Further, one could consider the following natural generalization to infinite graphs.

\begin{Problem}\label{prob3}
Is there an integer $c$ such that every infinite double tree $G$ has a \stf $(T_1,T_2)$ such that $|d_{T_1}(v)-d_{T_2}(v)|\leq c$ for all $v \in V(G)$?
\end{Problem}

Finally, a similar question can be asked in digraphs.

\begin{Problem}\label{prob4}
Is there an integer $c$ for which every digraph whose arc set can be factorized into two spanning $r$-arborescences for some $r \in V(G)$ can be factorized into two such arboresences $A_1,A_2$ satisfying $|d_{A_1}^+(v)-d_{A_2}^+(v)|\leq c$ for all $v \in V(G)$?
\end{Problem}
Clearly, Problems \ref{prob2}, \ref{prob3} and \ref{prob4} can also be asked in the more general setting when a factorization into $k$ objects for arbitrary $k$ is searched for. Also, for Problems \ref{prob3} and \ref{prob4}, approximate results would already be interesting. 
\medskip

We wish to mention that during the submission time of this paper, some progress on the above mentioned problems has been made by Illingworth, Powierski, Scott and Tamitegama in \cite{ipst}. In particular, negative answers for Problems \ref{prob2} and \ref{prob4} haven been provided and some progress on Problems \ref{prob5} and \ref{prob3} has been made.

\section*{Acknowledgements}
I wish to thank Daniel Gonçalves and Louis Esperet for pointing me to the idea of proving Conjecture \ref{ddfd} by applying Theorem \ref{main} multiple times. I further wish to thank Kalina Petrova for pointing out a mistake in an earlier version of this article. I would also like to thank Alberto Espuny D\'{i}az for the helpful discussions.


\begin{thebibliography}{99}
\bibitem{bhy} J. Bang-Jensen, F. Havet, A. Yeo,  The complexity of finding arc-disjoint branching flows, Discrete Applied Mathematics, {\bf 209}, 16-26, 2016.
\bibitem{bhmrs} S. Bessy, F. H\"orsch, A. K. Maia, D. Rautenbach, I. Sau,  FPT algorithms for packing $k$-safe spanning rooted sub(di)graphs, \url{https://arxiv.org/abs/2105.01582}.
\bibitem{cpt}J. Chuzhoy, M. Parter,  Z. Tan, On Packing Low-Diameter Spanning Trees, \url{https://arxiv.org/abs/2006.07486}.
\bibitem{book} A. Frank, Connections in Combinatorial Optimization, Oxford University Press, 2011.
\bibitem{ipst} F. Illingworth, E. Powierski, A. Scott,  Y. Tamitegama, Balancing connected colourings of graphs, \url{https://arxiv.org/abs/2205.04984v1}.
\bibitem{nk1} M. Kriesell, Balancing two spanning trees, Networks {\bf 57(4)}, 351-353, 2011.
\bibitem{l}F. Lehner, On spanning tree packings of highly edge connected graphs, Journal of Combinatorial Theory, Series B, {\bf 105}, 93-126, 2014.
\bibitem{nw}C.St.J.A. Nash-Williams, Edge-disjoint spanning trees of ﬁnite graphs, J. Lond. Math. Soc. {\bf 36}, 445–450, 1961.
\bibitem{s}M. Stein, Arboricity and tree-packing in locally finite graphs, Journal of Combinatorial Theory, Series B, {\bf 96}, 302-312, 2006.
\bibitem {t}W.T. Tutte, On the problem of decomposing a graph into $n$ connected factors, J. Lond. Math. Soc. {\bf 36}, 221–230, 1961.
\bibitem{w} W. Walter, Analysis 1, Springer, Berlin, 2004.


\end{thebibliography}
\end{document}